\documentclass{article}
\usepackage[utf8]{inputenc}
\usepackage[english]{babel}

\title{The Structure of Biquandle Brackets}
\author{Will Hoffer \\ \texttt{whoff003@ucr.edu} \and Adu Vengal \\ \texttt{vengal.8@osu.edu} \and Vilas Winstein \\ \texttt{winstein.1@osu.edu}}
\date{August 2019}

\usepackage[numbers]{natbib}
\usepackage{graphicx}
\usepackage{amssymb,amsmath,dsfont}
\usepackage{amsthm}
\usepackage{comment}

\usepackage{hyperref}
\hypersetup{
    colorlinks=true,
    linkcolor=blue,
    filecolor=blue,      
    urlcolor=blue,
    citecolor=blue,
}

\DeclareMathOperator{\ovtri}{\overline{\rhd}}
\DeclareMathOperator{\untri}{\underline{\rhd}}

\newtheorem{theorem}{Theorem}

\newtheorem{corollary}{Corollary}
\newtheorem{prop}{Proposition}

\theoremstyle{definition}
\newtheorem{definition}{Definition}
\newtheorem{example}{Example}
\newtheorem{remark}{Remark}

\renewcommand\qedsymbol{$\blacksquare$}

\begin{document}

\maketitle

\begin{abstract}
    In their paper entitled ``Quantum Enhancements and Biquandle Brackets,'' Nelson, Orrison, and Rivera introduced \emph{biquandle brackets}, which are customized skein invariants for biquandle-colored links.
    We prove herein that if a biquandle bracket $(A,B)$ is the pointwise product of the pair of functions $(A',B')$ with a function $\phi$, then $(A',B')$ is also a biquandle bracket if and only if $\phi$ is a a biquandle 2-cocycle (up to a constant multiple).
    As an application, we show that a new invariant introduced by Yang factors in this way, which allows us to show that the new invariant is in fact equivalent to the Jones polynomial on knots.
    Additionally, we provide a few new results about the structure of biquandle brackets and their relationship with biquandle 2-cocycles.
\end{abstract}

\section{Introduction}

\emph{Biquandles} are a type of algebraic structure whose axioms parallel the Reidemeister moves in knot theory.
Because of this, biquandles are the basis for many invariants of knots and links.
In particular, the \emph{biquandle counting invariant} is simply the number of ways to color a link diagram with elements of a biquandle so that relationships between colors at crossings are satisfied.
In \cite{nelson2017quantum}, an enhancement of the biquandle counting invariant was introduced, called the \emph{biquandle bracket}.
This is a type of skein relation depending on biquandle colorings.

A \emph{biquandle 2-cocycle} is another type of function on a biquandle that can be used to define a link invariant, arising from cohomology theory.
We found that if a biquandle bracket is the pointwise product of two functions, then one is a biquandle bracket if and only if the other is a 2-cocycle.
This result can be applied to a biquandle bracket proposed by Yang in \cite{yang2017enhanced}.
The biquandle bracket in question is such a pointwise product, wherein the biquandle bracket factor gives an invariant equivalent to the Jones polynomial.
The 2-cocycle factor gives a trivial invariant on knots, so this shows that Yang's invariant is altogether equivalent to the Jones polynomial for knots.

In some cases, a biquandle bracket is actually just a 2-cocycle in disguise.
We identify a sufficient condition for a biquandle to produce only this type of biquandle bracket, so that for this type of biquandle, the study of its brackets reduces to the study of its 2-cocycles (which are simpler in general).

This paper is structured as follows.
In section 2, we review definitions which we will require for our results, including the definition of biquandle brackets and biquandle 2-cocycles.
In section 3, we present a few results relating the structure of biquandle brackets to the structure of biquandle 2-cocycles.
In section 4, we present another result about the structure of the biquandle bracket, and provide some commentary on the result.
Finally, in section 5, we pose some questions for further research.

This work has been done as a part of the Summer 2018 undergraduate research program
\href{http://www.math.ohio-state.edu/~chmutov/wor-gr-su18/wor-gr.htm}{``\underline{Knots and Graphs}''}
at the Ohio State University.
We are grateful to the OSU Honors Program Research Fund and to the  NSF-DMS \#1547357 RTG grant: Algebraic Topology and Its Applications for financial support.
In addition, we are grateful to our advisor, Sergei Chmutov, for his help.

\section{Definitions and Notation}

To establish our notation and introduce the topics, we provide the following definitions.
We follow the notation and conventions in \cite{nelson2017quantum}.

\begin{definition}
    A \emph{biquandle} is a set $X$ with two binary operations $\untri, \ovtri$ such that $\forall  x, y, z \in X$,
    
    \begin{enumerate}
        \item[(i)]
            $x \untri x = x \ovtri x$
        
        \item[(ii)]
            The maps $\alpha_y(x) = x \ovtri y,\, \beta_y(x) = x \untri y$, and $S(x,y) = (y \ovtri x, x \untri y)$ are invertible. 
        
        \item[(iii)]
            The following exchange laws are satisfied:
            \begin{align*}
                (x\untri y) \untri (z \untri y) &= (x\untri z) \untri (y \ovtri z) \\
                (x\untri y) \ovtri (z \untri y) &= (x \ovtri z) \untri (y \ovtri z) \\
                (x\ovtri y) \ovtri (z \ovtri y) &= (x\ovtri z) \ovtri (y \untri z).
            \end{align*}
    \end{enumerate}
    
    If $x \ovtri y = x$ for all $x, y \in X$, then $X$ is called a \emph{quandle}.
    When there is no danger of confusion, we will write the biquandle $(X,\ovtri,\untri)$ simply as $X$.
\end{definition}

    If $X$ is a finite biquandle, we can represent all of the information about it in two operation tables.
    Fix some ordering on the elements of $X$ and label them with the integers $1$ through $n$ (where $n$ is the size of $X$).
    Then the operation table for $\untri$ is an $n \times n$ matrix of integers in $\{ 1, \dotsc, n \}$, and the $(i,j)$ entry of this matrix is $i \untri j$.
    The operation table for $\ovtri$ is defined similarly.
    For example, the following operation tables represent a biquandle on three elements.
    \begin{align*}
        \untri: \quad
        \begin{bmatrix}
            2 & 1 & 2 \\
            1 & 3 & 3 \\
            3 & 2 & 1
        \end{bmatrix}
        \qquad \qquad
        \ovtri: \quad
        \begin{bmatrix}
            3 & 3 & 3 \\
            2 & 2 & 2 \\
            1 & 1 & 1
        \end{bmatrix}
    \end{align*}

    The conditions in the biquandle definition are analogous to the Reidemeister moves in knot theory when we interpret $x \untri y$ as ``$x$ passing under $y$'' and $x \ovtri y$ as ``$x$ passing over $y$'' in the following way:
    
    \begin{center}
        \includegraphics[scale=0.15]{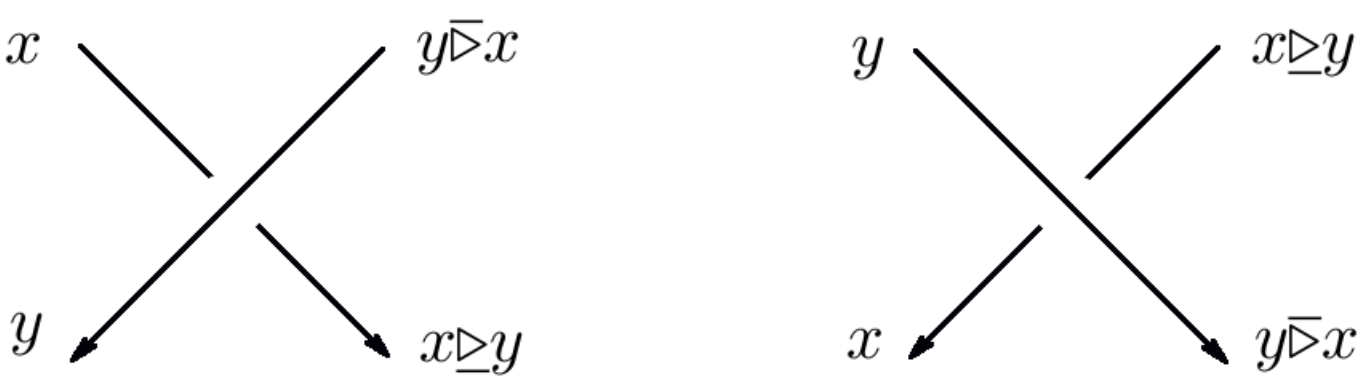}
    \end{center}
    
    Fix a biquandle $X$.
    An $X$-coloring of an oriented knot (or link) diagram $D$ is an assignment of an element of $X$ to each strand in the diagram such that the above relationships hold at each crossing.
    Then the biquandle axioms are precisely what is required for the $X$-coloring to be preserved as Reidemeister moves are performed on the diagram.
    For this reason, the number of $X$-colorings of a diagram is a link invariant, called the \emph{biquandle counting invariant}.
    
    In \cite{nelson2017quantum}, an enhancement of the biquandle counting invariant is introduced.
    For each $X$-coloring of $D$, one can perform a smoothing operation similar to the construction in the Kauffman bracket, but this time keeping track of the colorings at each crossing as follows:
    
    \begin{center}
        \includegraphics[scale=0.22]{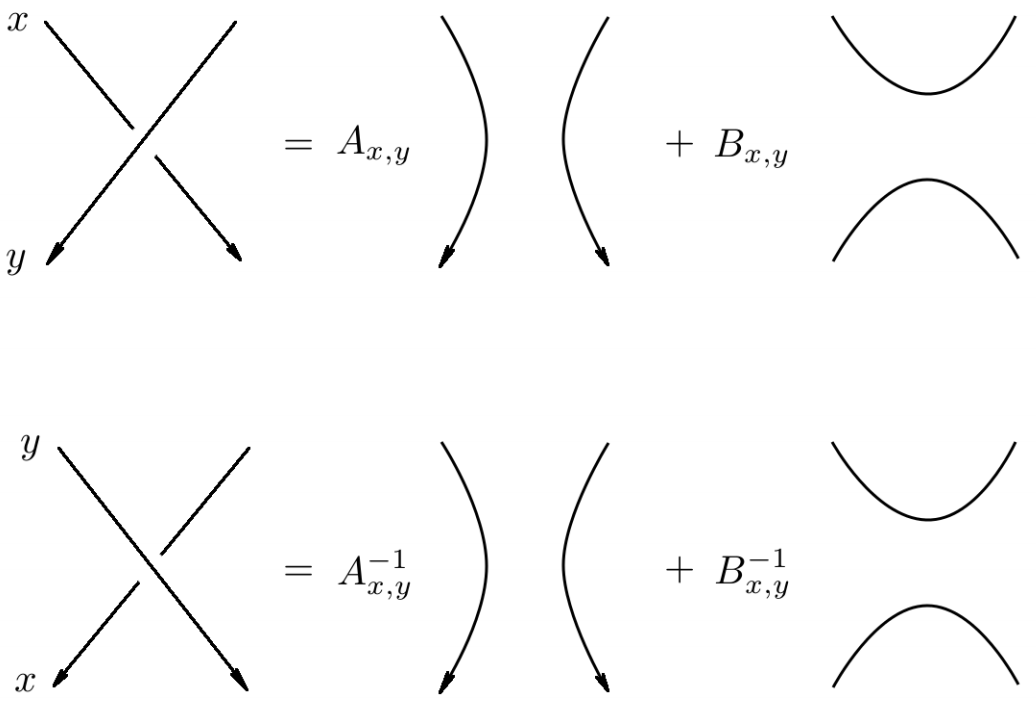}
    \end{center}
    
    Where for each $x,y \in X$, $A_{x,y}$ and $B_{x,y}$ are invertible elements of some commutative ring with unity $R$.
    Additionally, the removal of a circle with no crossings should correspond to multiplication by some element $\delta \in R$, and to correct for the additional states generated by kinks (from the first Reidemeister move), a writhe factor should be included, which can simply be an appropriate power of some element $w \in R^\times$.
    For the bracket to be an invariant of an $X$-colored link, it should not change when Reidemeister moves are applied and the $X$-coloring is updated correspondingly.
    Below are the conditions that must be satisfied by $A,B,\delta$, and $w$ for this to be true.
    For more details, see \cite{nelson2017quantum}.
 
\begin{definition}
    A \emph{biquandle bracket} on a biquandle $X$ with values in commutative ring (with unity) $R$ is a pair of maps $A,B:X\times X \rightarrow R^\times$ and two distinguished elements $\delta\in R, w\in R^\times$ which satisfy the following conditions.
    
    \begin{enumerate}
        \item[(i)]
            For all $x \in X$,
            $\delta A_{x,x}+B_{x,x}=w$ and $\delta A_{x,x}^{-1}+B_{x,x}^{-1}=w^{-1}$.
            
        \item[(ii)]
            For all $x, y \in X$,
            $\delta = -A_{x,y}B_{x,y}^{-1}-A_{x,y}^{-1}B_{x,y}$.
            
        \item[(iii)]
            For all $x, y, z \in X$, all of the following equations hold.
            \begin{align*}
                A_{x,y} A_{y,z} A_{x \untri y, z \ovtri y} &= A_{x,z} A_{y \ovtri x, z \ovtri x} A_{x \untri z, y \untri z}, \\
                A_{x,y} B_{y,z} B_{x \untri y, z \ovtri y} &= B_{x,z} B_{y \ovtri x, z \ovtri x} A_{x \untri z, y \untri z}, \\
                B_{x,y} A_{y,z} B_{x \untri y, z \ovtri y} &= B_{x,z} A_{y \ovtri x, z \ovtri x} B_{x \untri z, y \untri z}, \\
                A_{x,y} A_{y,z} B_{x \untri x, z \ovtri y} &= A_{x,z} B_{y \ovtri x, z \ovtri z} A_{x \untri z, y \untri z} \\
                &\qquad + A_{x, z} A_{y \ovtri x, z \ovtri x} B_{x \untri z, y \untri z} \\
                &\qquad + \delta A_{x, z} B_{y \ovtri x, z \ovtri x} B_{x \untri z, y \untri z} \\
                &\qquad + B_{x, z} B_{y \ovtri x, z \ovtri x} B_{x \untri z, y \untri z}, \\
                B_{x, z} A_{y \ovtri x, z \ovtri x} A_{x \untri z, y \untri z} &= B_{x, y} A_{y, z} A_{x \untri y, z \ovtri y} \\
                &\qquad + A_{x, y} B_{y, z} A_{x \untri y, z \ovtri y} \\
                &\qquad + \delta B_{x, y} B_{y, z} A_{x \untri y, z \ovtri y} \\
                &\qquad + B_{x, y} B_{y, z} B_{x \untri y, z \ovtri y}.
            \end{align*}
    \end{enumerate}
    
    Note that we denote $A(x,y)$ and $B(x,y)$ by $A_{x,y}$ and $B_{x,y}$.
    Additionally, since $\delta$ and $w$ are determined by the maps $A$ and $B$, we will generally denote a biquandle bracket simply by the pair $(A,B)$.
    Finally, if $(A,B)$ is a biquandle bracket on a biquandle $X$ taking values in $R$, then we say $(A,B)$ is an \emph{$X$-bracket}.
\end{definition}

The oriented link invariant corresponding to $(A,B)$ is simply the multiset of all biquandle bracket values, one for each valid $X$-coloring of the diagram.
This is an enhancement of the biquandle counting invariant because the counting invariant is simply the cardinality of this multiset.

If $X$ is finite and we fix an ordering $X = \left\{ x_1, \dotsc, x_n \right\}$, we can encapsulate all of the information about a biquandle bracket in a presentation matrix.
This is an $n$ by $2n$ matrix $M$ over $R$ with entries $M_{i,j}=A_{x_i,x_j}$ and $M_{i,n+j}=B_{x_i,x_j}$ for $i,j\in\{1,2,...,n\}$.

\begin{example}
    Let $X$ be the biquandle given by the following operation table.
    \begin{align*}
        \untri: \quad
        \begin{bmatrix}
            2 & 2 \\
            1 & 1 \\
        \end{bmatrix}
        \qquad \qquad
        \ovtri: \quad
        \begin{bmatrix}
            2 & 2 \\
            1 & 1 \\
        \end{bmatrix}
    \end{align*}
    This biquandle's operations simply flip the left operand, regardless of the right operand.
    Thus, in any $X$-colored link diagram, if one follows a particular strand, the color will alternate at every crossing.
    Let $R = \mathbb{Z}_2[t] / \left( 1 + t + t^3 \right)$.
    Then the following presentation matrix defines an $X$-bracket.
    \begin{align*}
        \left[
            \begin{array}{cc|cc}
                1 & 1+t & t & t+t^2 \\
                1+t^2 & 1 & 1 & t \\
            \end{array}
        \right]
    \end{align*}
    This biquandle bracket was found in \cite{nelson2017quantum}.
\end{example}

\begin{example}
    Let $X$ be any biquandle, and let $R$ be any commutative unital ring.
    If $A_{x,y} = a$ and $B_{x,y} = b$ for all $x,y \in X$ and some $a,b \in R^\times$, then the $X$-bracket $(A,B)$ is called a ``constant'' biquandle bracket (the reader should verify that this does indeed define an $X$-bracket).
    In general, the value of a biquandle bracket on links is unchanged when all values of $A_{x,y}$ and $B_{x,y}$ are scaled by a common factor of $R^\times$ (see \cite{nelson2017quantum}).
    So, dividing through by $b$, the above bracket gives the same invariant as the bracket $A_{x,y} = \frac{a}{b}$, $B_{x,y} = 1$ for all $x,y \in X$. By considering the maps $A, B$ as instead taking values in $R\left[\left(\frac{a}{b}\right)^{\pm 1/2}\right]$, we can make the substitution $q^2 = \frac{a}{b}$ and divide everything through by $q$ to yield the equivalent bracket (when treated over $R\left[\left(\frac{a}{b}\right)^{\pm 1/2}\right]$), $A_{x,y} = q$, $B_{x,y} = q^{-1}$ for all $x, y \in X$. 
    Now, for any particular $X$-coloring of a link, the value of this bracket is evidently the Jones polynomial of the link, which is a Laurent polynomial in the variable $q^2$. Hence the invariant itself still takes values in $R$ rather than $R\left[\left(\frac{a}{b}\right)^{\pm 1/2}\right]$.
    Therefore, the value of a constant biquandle bracket (with $A_{x,y} = a$ and $B_{x,y} = b$) is a multiset containing the Jones polynomial evaluated at $\frac{a}{b}$, and it contains this value with multiplicity equal to the number of valid $X$-colorings of the link.
\end{example}

Next, we define a biquandle 2-cocycle following the notation of \cite{nelson2017quantum}.
 
\begin{definition}
    Let $X$ be a biquandle, and let $G$ be an abelian group (written multiplicatively here).
    A function $\phi:X\times X\rightarrow G$ is a \emph{biquandle 2-cocycle} if, for all $x, y, z \in X$, we have
    \begin{enumerate}
        \item[(i)] $\phi(x,x) = 1$,
        \item[(ii)] $\phi(x,y) \cdot \phi(y,z) \cdot \phi\left(x\untri y, z\ovtri y\right) = \phi(x,z) \cdot \phi\left(y\ovtri x,z\ovtri x\right) \cdot \phi\left(x\untri z, y\untri z\right)$.
    \end{enumerate}
\end{definition}

A $2$-cocycle $\phi$ can be used to define the \emph{biquandle 2-cocycle invariant}, as seen in \cite{ceniceros2014augmented}.
Namely, for each valid $X$-coloring of a link, compute the value $\prod_{\tau} \phi\left(x_\tau,y_\tau\right)^{\epsilon(\tau)}$, where $\tau$ ranges across all crossings in the colored link, $\epsilon(\tau)$ is the sign (either $+1$ or $-1$) of $\tau$, and $x_\tau,y_\tau$ are the biquandle colors of the arcs on the left side of the crossing when it is oriented so that strands point downwards, following a similar convention to the biquandle bracket above.
The value of the biquandle $2$-cocycle invariant associated to $\phi$ is then the multiset of all such values, one for each valid $X$-coloring of the link.

Again if $X$ is finite, we construct a presentation matrix for a cocycle in the same fashion as with the biquandle brackets; fixing the ordering $X = \left\{ x_1, \dotsc, x_n \right\}$, the presentation matrix $P$ for a cocycle is an $n \times n$ matrix over $A$ with entries $P_{i,j} = \phi \left( x_i, x_j \right)$.

\begin{example}
    Let $X$ be the biquandle described in Example 1 above.
    Let $A$ be the free abelian group on two symbols, $a$ and $b$.
    Then the following presentation matrix defines a biquandle 2-cocycle $\phi : X \times X \to A$.
    \begin{align*}
        \begin{bmatrix}
            1 & a \\
            b & 1
        \end{bmatrix}
    \end{align*}
    The invariant corresponding to $\phi$ is trivial on all knots (single-component links).
    In fact, more is true: for any $X$-colored knot diagram, at any crossing $\tau$, we have $x_\tau = y_\tau$ (so that $\phi\left(x_\tau, y_\tau\right) = 1$).
    To see this, consider any crossing $\tau$ in the diagram, oriented downward.
    Follow the strand starting at the bottom-left arc of the $\tau$.
    When this strand first returns to $\tau$, it must connect to the top-left arc of $\tau$.
    If it connected to the top-right arc first, then it would close the loop and the diagram would have more than one component (and so not be a knot).
    If it connected to the bottom-right arc first, then the orientation of the strand would be inconsistent.
    
    This strand now makes a closed loop to the left of $\tau$.
    Any time this closed loop intersects itself in a crossing, the strand must pass through this crossing twice.
    The rest of the knot diagram (excluding this closed loop to the left of $\tau$) makes a closed loop to the right of $\tau$.
    And each time this other closed loop crosses the first closed loop, it must cross back at some point, since it must end up on the same side (inside or outside) of the closed loop that it started in.
    Thus the number of crossings that the strand starting from the bottom-left arc of $\tau$ encounters before it gets to the top-left arc of $\tau$ is even.
    Therefore in any $X$-coloring, since the color changes at each crossing and there are two colors, $x_\tau$ must be the same as $y_\tau$.
    
    For an $k$-component link diagram, there are exactly $2^k$ $X$-colorings: simply pick a color for some arc of some component (a binary choice) and walk along that component, switching the color at each crossing.
    The component will be involved in an even number of crossings (counting self-crossings twice), so this procedure will terminate consistently.
    The above argument shows that the value of the invariant for knots corresponding to $\phi$ is exactly the multiset $\left\{ 1, 1 \right\}$.
    It is not hard to see, by modifying the above argument, that for a two-component link, the invariant corresponding to $\phi$ is the multiset $\left\{ 1, 1, \left( a b \right)^\ell, \left( a b \right)^\ell \right\}$, where $\ell$ is the linking number of the two components of the link.
    For links with more components, the invariant's behavior is more complicated.
\end{example}

\section{Results}

\begin{theorem}
    Let $X$ be a biquandle, and let $(A,B)$ be an $X$-bracket over a ring $R$.
    Suppose that there exist functions $A', B', \phi: X \times X \rightarrow R^\times$ such that for every $x,y\in X$ we have $A_{x,y}=A'_{x,y} \cdot \phi(x,y)$ and $B_{x,y}=B'_{x,y} \cdot \phi(x,y)$.
    Then $(A',B')$ form a biquandle bracket if and only if $\phi$ is a biquandle $2$-cocycle, up to a constant multiple.
\end{theorem}

\begin{proof}
    We can write the first equation of biquandle bracket condition (iii) as
    \begin{multline}
        A'_{x,y}A'_{y,z}A'_{x\ovtri y,z \untri y}\phi(x,y)\phi(y,z)\phi(x\ovtri y,z \untri y) \\
        =A'_{x,z}A'_{y\ovtri x,z \ovtri x}A'_{x\untri z,y \untri z}\phi(x,z)\phi(y\ovtri x,z \ovtri x)\phi(x\untri z,y \untri z).
    \end{multline}
    Hence we have 
    $$ A'_{x,y}A'_{y,z}A'_{x\ovtri y,z \untri y}=A'_{x,z}A'_{y\ovtri x,z \ovtri x}A'_{x\untri z,y \untri z} $$
   if and only if 
    \begin{equation}
        \phi(x,y)\phi(y,z)\phi(x\ovtri y,z \untri y)=\phi(x,z)\phi(y\ovtri x,z \ovtri x)\phi(x\untri z,y \untri z),
    \end{equation}
    which is 2-cocycle condition (ii). It's easy to see that the 2-cocycle condition (ii) similarly factors out of all the other equations in biquandle bracket condition (iii). \\
    For any $x,y \in X$ we have  $$ -A_{x,y} B_{x,y}^{-1} - A_{x.y}^{-1} B_{x,y} = \delta = -A'_{x,y} \left( B'_{x,y} \right)^{-1} - \left( A'_{x,y} \right)^{-1} B'_{x,y}$$
    Thus biquandle bracket condition (ii) is satisfied regardless of $\phi$.
    
    \noindent As a special case of biquandle bracket condition (ii), we see that
    $$ \delta = -A_{x,x} B_{x,x}^{-1} - A_{x.x}^{-1} B_{x,x} $$
    Plugging this in to biquandle bracket condition (i) we get
    $$ w=-A_{x,x}^2B_{x,x}^{-1}. $$
    This is true for every $x\in X$ so for all $x,y \in X$,
    $$ -A_{x,x}^2B_{x,x}^{-1}=-A_{y,y}^2B_{y,y}^{-1}. $$
    It follows then that 
    $$ - \left( A'_{x,x} \right)^2 \left( B'_{x,x} \right)^{-1} = - \left( A'_{y,y} \right)^2 \left( B'_{y,y} \right)^{-1}. $$
    if and only if
    $\phi(x,x)=\phi(y,y)$.
\end{proof}

Note that the constant factor in $\phi$ can essentially be cancelled out at no cost, since it can be absorbed into the biquandle bracket $(A',B')$ and since biquandle brackets differing by constants define the same invariant.

\begin{remark}
    Suppose an $X$-bracket $(A,B)$ factors in this way to a pointwise product of $\left( A', B' \right)$ and $\phi$.
    Then the value of $(A,B)$ on an $X$-colored link will be the value of $\left( A', B' \right)$ multiplied by the value of the biquandle $2$-cocycle invariant associated with $\phi$.
    Thus, if we retain the information about which coloring was associated with each value of the bracket and $2$-cocycle invariant, then the invariant defined by $(A,B)$ cannot be more powerful than the invariant defined by $\left( A', B' \right)$ and the $2$-cocycle invariant defined by $\phi$, computed in tandem. 
    
    Since cohomologous $2$-cocycles define the same $2$-cocycle invariant \cite{ceniceros2014augmented}, this remark also gives an alternate proof of proposition 2 in \cite{nelson2017quantum}, showing that biquandle brackets differing by coboundaries define the same invariant.
\end{remark}

\begin{example}
    In \cite{yang2017enhanced}, Yang introduced a new biquandle bracket which is a generalization of the bracket in example 1, introduced in \cite{nelson2017quantum}.
    The underlying biquandle is the two-element set $X = \mathbb{Z}_2$ with the operations being $x \ovtri y = x \untri y = 1 - x$ (the action `flips' the left argument, so an $X$-coloring of a link can be viewed as a $2$-coloring where the color of a strand changes at every crossing).
    This biquandle is the same one presented in Examples 1 and 3.
    Yang's bracket takes values in any commutative ring $R$ with unity.
    Let $a,b,n,e,w \in R^\times$ be arbitrary invertible elements.
    Then the bracket is given by the following matrix, using the notation introduced after Definition 2 above with $x_1 = 0, x_2 = 1$.
    \begin{align*}
        \left[
            \begin{array}{cc|cc}
                na & ea & nb & eb \\
                wa & na & wb & nb \\
            \end{array}
        \right]
    \end{align*}
    This matrix is the following Hadamard (entry-wise) product of the following two matrices.
    \begin{align*}
        \left[
            \begin{array}{cc|cc}
                na & na & nb & nb \\
                na & na & nb & nb \\
            \end{array}
        \right]
        \odot
        \left[
            \begin{array}{cc|cc}
                1 & n^{-1}e & 1 & n^{-1}e \\
                n^{-1}w & 1 & n^{-1}w & 1  \\
            \end{array}
        \right]
    \end{align*}
    Now the left multiplicand is the presentation for the constant $X$-bracket $A',B'$ having $A_{x,y}' = na$ and $B_{x,y}' = nb$ for all $x,y \in X$.
    Thus, by Theorem 1, the function defined by $\phi(0,0) = \phi(1,1) = 1$, $\phi(0,1) = n^{-1} e$, and $\phi(1,0) = n^{-1} w$ is a biquandle $2$-cocycle (up to a constant multiple).
    In this case, $\phi(x,x) = 1$ for all $x \in X$, so $\phi$ is itself a $2$-cocycle.

    Thus, as in Example 2, the value of the bracket $\left( A', B' \right)$ on a knot is simply a multiset containing two copies of the Jones polynomial evaluated at $\frac{a}{b}$.
    So, in general, the value of $\left( A', B' \right)$ on a link is the Jones polynomial with multiplicity $2^k$ (recall from Example 3 that there are exactly $2^k$ $X$-colorings of a $k$-component link).
    Additionally, $\phi$ is recognized to be the same $2$-cocycle as the one presented in Example 3.
    So, since the invariant corresponding to $\phi$ is trivial on knots, Yang's invariant is equivalent to the Jones polynomial on knots.
\end{example}

This example suggests the following proposition:
\begin{prop}
Let $(A,B)$ be an $X$-bracket, where the ratio $A_{x,y}B_{x,y}^{-1}$ is constant. Then $(A,B)$ is the product of a constant bracket and a 2-cocycle. 
\end{prop}

\begin{proof}
    Choose $x_0\in X$ and let $a = A_{x_0,x_0}, b = B_{x_0,x_0}$. Then for any $x,y\in X$, we have  $A_{x,y}= a \phi(x,y)$ and $B_{x,y}= b \phi(x,y)$ for some $\phi:X\times X \rightarrow R^\times$. By Theorem 1, $\phi$ is a 2-cocycle.
\end{proof}

\begin{remark}
    Hence the value of any such $X$-bracket on an $X$-colored link factors into the product of the Jones Polynomial evaluated at $\frac{a}{b}$ and the 2-cocycle invariant defined by $\phi$.
    Thus to create $X$-brackets that distinguish links differently from the Jones polynomial and 2-cocycles, we would want the function $(x,y) \mapsto B_{x,y}^{-1} A_{x,y}$ take on more than one value.
    We next present a result and some examples concerning the number of values that can be taken by this function.
\end{remark}

\begin{theorem}
    Let $X$ be a biquandle, and let $(A,B)$ be an $X$-bracket.
    Fix some $x_0, y_0 \in X$ and let $a = A_{x_0,y_0}$ and $b = B_{x_0,y_0}$.
    If $R$ is an integral domain, then there exists some function $\psi: X \times X \to R^\times$ such that, for each $x,y\in X$, one of the following conditions holds.
    \begin{itemize}
        \item[\emph{(i)}]
            $A_{x,y} = a \cdot \psi(x,y)$ and $B_{x,y} = b \cdot \psi(x,y)$.
        
        \item[\emph{(ii)}]
            $A_{x,y} = b \cdot \psi(x,y)$ and $B_{x,y} = a \cdot \psi(x,y)$.
    \end{itemize}
\end{theorem}

\begin{proof}
    For any $x,y \in X$, biquandle bracket condition (ii) gives 
    $$ -ab^{-1}-ba^{-1}=-A_{x,y}B_{x,y}^{-1}-A_{x,y}^{-1}B_{x,y} $$
    If we let $u=ab^{-1}$, $v=A_{x,y}B_{x,y}^{-1}$, then this is
    \begin{align*}
        u+u^{-1} &= v+v^{-1} \\
        u^2v+v &= v^2u+u \\
        uv(v-u)-(v-u) &= 0 \\
        (uv-1)(v-u) &= 0
    \end{align*}
    Therefore we either have $ab^{-1}=A_{x,y}B_{x,y}^{-1}$, in which case we have
    \[
        A_{x,y}a^{-1}=B_{x,y}b^{-1}=:\psi(x,y),
    \]
    or we have $ab^{-1}=B_{x,y}A_{x,y}^{-1}$, so that 
    \[
        A_{x,y}b^{-1}=B_{x,y}a^{-1}=:\psi(x,y). \pushQED{} \tag*{\qedsymbol}
    \]
\end{proof}

\begin{remark}
    This result shows that, under the hypothesis conditions, $B_{x,y}^{-1} A_{x,y}$ can only possibly take two values: $b^{-1} a$ or $a^{-1} b$.
    The result cannot be strengthened by removing the possibility of condition (ii) and thus concluding that $B_{x,y}^{-1} A_{x,y}$ is constant in $x$ and $y$.
    
    To see this, consider the following biquandle bracket found in \cite{nelson2017trace}.
    The biquandle $X$ for the bracket has three elements, and the operation tables for $\untri$ and $\ovtri$ are as follows:
    \begin{align*}
        \untri: \quad
        \begin{bmatrix}
            3 & 1 & 3 \\
            2 & 2 & 2 \\
            1 & 3 & 1
        \end{bmatrix}
        \qquad \qquad
        \ovtri: \quad
        \begin{bmatrix}
            3 & 3 & 3 \\
            2 & 2 & 2 \\
            1 & 1 & 1
        \end{bmatrix}
    \end{align*}
    The $X$-bracket $(A,B)$ takes values in $\mathbb{Z}_5$, having the following presentation matrix.
    \begin{align*}
         \left[
            \begin{array}{ccc|ccc}
                1 & 3 & 1 & 2 & 4 & 2 \\
                1 & 4 & 1 & 2 & 2 & 2 \\
                1 & 3 & 1 & 2 & 4 & 2
            \end{array}
        \right]
    \end{align*}
    Notice that $B_{1,1}^{-1} A_{1,1} = 1^{-1} 2 = 2$, whereas $B_{1,2}^{-1} A_{1,2} = 3^{-1} 4 = 2 \cdot 4 = 3 \neq 2$.
    This example shows the existence of biquandle brackets taking values in an integral domain for which $B_{x,y}^{-1} A_{x,y}$ is not constant in $x$ and $y$.
\end{remark}

\begin{corollary}
    Let $X$ be a biquandle, and let $(A,B)$ be an $X$-bracket taking values in an integral domain $R$. Suppose there exists $x_0,y_0 \in X$ such that $A_{x_0,y_0} = B_{x_0,y_0}$. Then $A_{x,y} = B_{x,y}$ for all $x,y \in X$, and so each of the functions $A,B$ are a 2-cocycle, up to a constant multiple.
\end{corollary}

\begin{proof}
    Let $x,y \in X$. In either of the cases in theorem 2, we have $A_{x,y} = A_{x_0,y_0} \psi(x,y) = B_{x,y}$, and whenever $A_{x,y} = B_{x,y}$ for all $x,y \in X$, each of $A,B$ form a 2-cocycle, up to a constant multiple (this was shown in \cite{nelson2017quantum} and also follows from an easy application of theorem 1).
\end{proof}

\begin{remark}
    The above result is not true in general for commutative rings which are not integral domains.
    
    For example, consider the following biquandle bracket found using a computer program.
    The biquandle $X$ for the bracket is the so-called ``trivial biquandle'' on two elements (call them $1$ and $2$).
    This means that $x \untri y = x \ovtri y = x$ for all $x, y \in X = \{ 1, 2 \}$.
    The $X$-bracket $(A, B)$ takes values in $\mathbb{Z}_4$, having the following presentation matrix.
    \begin{align*}
         \left[
            \begin{array}{cc|cc}
                1 & 1 & 1 & 3 \\
                1 & 1 & 3 & 1 \\
            \end{array}
        \right]
    \end{align*}
    Notice that $A_{1,1} = B_{1,1}$ but $A_{1,2} \neq B_{1,2}$.
\end{remark}

\begin{remark}
    By Theorem 2, the function $(x,y) \mapsto B_{x,y}^{-1} A_{x,y}$ can only take two possible values whenever $R$ is an integral domain.
    In the remark above, $R$ is not an integral domain but this function still only takes two values.
    However, in general, this function may take more than two values.
    
    In the remark above, the function $(x,y) \mapsto B_{x,y}^{-1} A_{x,y}$ can only take two possible values.
    By Theorem 2, this is also true whenever $R$ is an integral domain.
    However, this is not necessarily true if $R$ is not an integral domain.
    
    In the remark above and in all cases where $R$ is an integral domain, the function $(x,y) \mapsto B_{x,y}^{-1} A_{x,y}$ can only take two possible values.
    This is also not true in a general ring which is not an integral domain.
    
    For example, consider the following biquandle bracket (again found using a computer program) over the trivial biquandle on three elements, taking values in the ring $\mathbb{Z}_9$.
    \begin{align*}
         \left[
            \begin{array}{ccc|ccc}
                1 & 1 & 1 & 1 & 1 & 1 \\
                1 & 1 & 1 & 4 & 1 & 4 \\
                1 & 1 & 1 & 4 & 7 & 1
            \end{array}
        \right]
    \end{align*}
    Notice that $B_{3,1}^{-1} A_{3,1} = 4^{-1} 1 = 7$, $B_{3,2}^{-1} A_{3,2} = 7^{-1} 1 = 4$, and $B_{3,3}^{-1} A_{3,3} = 1^{-1} 1 = 1$, so there are three possible values of $B_{x,y}^{-1} A_{x,y}$.
\end{remark}

\begin{remark}
    Theorem $1$ says that if we stay in the first case of the result of Theorem $2$, i.e. that if $A_{x,y} = a \cdot \psi(x,y)$ and $B_{x,y} = b \cdot \psi(x,y)$ for all $x,y \in X$, then $\psi$ is a biquandle $2$-cocycle.
    This is because the functions $A'_{x,y} = a$ and $B'_{x,y} = b$ constitute the constant bracket $\left( A', B' \right)$.
    One might hope that the function $\psi$ is always a constant multiple of a biquandle 2-cocycle, even if both cases in the result occur.
    However, this is not true in general.
    For example, consider the same biquandle bracket as above in Remark 3.
    
    Take $x_0 = y_0 = 1$ so that $a = A_{1,1} = 1$ and $b = B_{1,1} = 2$.
    Now we have $A_{2,2} = 4$ and $B_{2,2} = 2$, so we are clearly not in case (i), and thus $A_{2,2} = 2 \cdot \psi(2,2)$ and $B_{2,2} = 1 \cdot \psi(2,2)$, so $\psi(2,2) = 2$.
    By construction, the function $\psi$ in the theorem satisfies $\psi\left( 1, 1 \right) = 1$.
    Thus $\psi$ is not constant on the diagonal subset of $X \times X$, which shows that it cannot be a constant multiple of a 2-cocycle.
\end{remark}

\begin{example}
    In all of the above examples of biquandles, all values of $B_{x,y}^{-1} A_{x,y}$ are of finite order in $R^\times$.
    This is not true in general.
    For example, consider the following biquandle bracket over the trivial biquandle on $2$ elements, taking values in $\mathbb{Z}\left[t, t^{-1} \right]$.
    \begin{align*}
        \left[
        \begin{array}{cc|cc}
            1 & 1 & t & t \\
            1 & 1 & t^{-1} & t
        \end{array}
        \right]
    \end{align*}
    Notice that $B_{x,y}^{-1} A_{x,y} = t$ or $t^{-1}$, both of which have infinite multiplicative order.
\end{example}

\section{More about Biquandle Brackets and 2-Cocycles}

We next present a result about the behavior of biquandle brackets on the diagonal subset of $X \times X$, i.e. the set $\left\{ (x,x) : x \in X \right\}$.

\begin{theorem}
    Let $(A,B)$ be an $X$-bracket. Then for any $x,y \in X$ we have
    \begin{itemize}
        \item[\emph{(i)}]
            $A_{x,x} = A_{x \untri y, x \untri y} = A_{x \ovtri y, x \ovtri y}$.
        
        \item[\emph{(ii)}]
            $B_{x,x} = B_{x \untri y, x \untri y} = B_{x \ovtri y, x \ovtri y}$.
    \end{itemize}
\end{theorem}

\begin{proof}
    Taking $z=y$ in biquandle bracket condition (iii) gives 
    \begin{equation*}
        A_{x,y} A_{y,y} A_{x \untri y, y \ovtri y} = A_{x,y} A_{y \ovtri x, y \ovtri x} A_{x \untri y, y \untri y}
    \end{equation*}
    And since $y \untri y = y \ovtri y$, this reduces to $A_{y,y} = A_{y \ovtri x, y \ovtri x}$. Taking $y=x$ instead gives 
    \begin{equation*}
         A_{x,x} A_{x,z} A_{x \untri x, z \ovtri x} = A_{x,z} A_{x \ovtri x, z \ovtri x} A_{x \untri z, x \untri z}
    \end{equation*}
    which similarly reduces to $A_{x,x} = A_{x \untri z, x \untri z}$. Since for any $x,y \in X$, biquandle bracket conditions (i) and (ii) yield $A_{x,x}^2B_{x,x}^{-1}=A_{y,y}^2B_{y,y}^{-1}$, we then also have $B_{x,x} = B_{x \untri y, x \untri y} = B_{x \ovtri y, x \ovtri y}$.
\end{proof}

Given an $X$-bracket $(A,B)$, if $A_{x,x}=A_{y,y}$ for every $x,y \in X$, then the function $A$ is a $2$-cocycle (up to a constant multiple). The previous theorem then suggests a class of biquandles that always give brackets with this property.

\begin{definition}
    A biquandle $X$ is semi-transitive if there exists $x \in X$ such that $x \untri X \bigcup x \ovtri X = X$ (where $x \untri X = \{ x \untri y$ : $y \in X \}$, and likewise for $x \ovtri X$).
\end{definition}

\begin{corollary}
    Let $X$ be a semi-transitive biquandle, and let $(A,B)$ be an $X$-bracket.
    Then $A$ is a 2-cocycle (up to a constant multiple).
\end{corollary}

\begin{proof}
    Let $x \in X$ be the element described in the definition of semi-transitivity.
    Then for any $y \in X$, either $y = x \ovtri z$ for some $z \in X$ or $y = x \untri z$ for some $z \in X$.
    In either case, $A_{y,y} = A_{x,x}$ since
    $$ A_{x \ovtri z, x \ovtri z} = A_{x,x} = A_{x \untri z, x \untri z}. $$
\end{proof}
    It follows that if $X$ is semi-transitive, then any $X$-bracket is the product of a bracket $(\textbf{1},B)$, a 2-cocycle, and a constant multiple, where $\textbf{1}:X\times X \rightarrow R^\times$ is defined by $\textbf{1}(x,y) = 1$ for all $x,y \in X$. Since the link invariant itself is unchanged by constant multiples, when constructing $X$-brackets it's enough to chose a 2-cocycle and a function $B:X\times X \rightarrow R^\times$ satisfying the simpler set of biquandle bracket axioms given by setting $A_{x,y} = 1$ for all $x,y\in X$
\begin{remark}
    The definition and result above would be of no substance if there were no semi-transitive biquandles.
    However there are many.
    In particular, all odd-degree dihedral quandles are semi-transitive.
    The dihedral quandle of degree $n$ is a quandle structure on the set $\mathbb{Z}_n$ where $x \untri y = 2 y - x$ (and, since it is a quandle, $x \ovtri y = x$) for all $x, y \in \mathbb{Z}_n$.
    If $n$ is odd, then for any $x, z \in \mathbb{Z}_n$, there is a unique $y$ such that $2 y - x = z$.
    Namely, $y = 2^{-1} (x + z)$ (here we see why $n$ must be odd---$2$ must be invertible).
    Thus, in fact, any $x \in \mathbb{Z}_n$ can be the element described in the definition of semi-transitivity.
\end{remark}

\section{Further Questions}

    We conclude with some avenues for further inquiry. In Theorem 2, what more can be said about the function $\psi$?
    For instance, can it be made into a knot invariant? 
    If we split $\psi$ into two separate functions based on each case, i.e. either $A_{x,y}=a\phi(x,y)$ or $A_{x,y} = b\psi(x,y)$, then is $\phi$ a cocycle?
    For what biquandle brackets will $\psi$ be a cocycle?
    In addition, what other properties of biquandle brackets can be found when restricted to specific classes of biquandles or rings.

\nocite{elhamdadi2015quandles}
\bibliographystyle{plain}
\bibliography{main}

\end{document}